\documentclass[11pt]{article}
\usepackage[T1]{fontenc}
\usepackage[utf8]{inputenc}
\usepackage{lmodern}
\usepackage{amsmath,amssymb,amsthm,mathtools}
\usepackage[a4paper,margin=28mm]{geometry}
\usepackage{enumitem}
\usepackage{authblk}
\usepackage[colorlinks=true,linkcolor=blue,citecolor=blue,urlcolor=blue]{hyperref}
\usepackage[backend=biber,style=numeric,sorting=nty,doi=false,isbn=false,
  url=false,eprint=false,maxbibnames=100,giveninits=true]{biblatex}
\renewbibmacro{in:}{}

\newcommand{\scrN}{\mathcal N}
\newcommand{\Nho}{\scrN^h_\Omega}
\newcommand{\Haus}{\mathcal H}
\newcommand{\frakb}{\mathfrak b}
\newcommand{\frakd}{\mathfrak d}
\newcommand{\frakc}{\mathfrak c}
\newcommand{\frakv}{\mathfrak v}
\newcommand{\Q}{\mathbb Q}
\newcommand{\Pp}{\mathbb P}
\newcommand{\R}{\mathbb R}
\newcommand{\restrict}{\upharpoonright}
\newcommand{\append}{\mathbin{{}^\frown}}
\newcommand{\seq}[1]{\langle #1\rangle}
\newcommand{\Con}{\operatorname{Con}}
\newcommand{\ZFC}{\mathsf{ZFC}}
\newcommand{\CH}{\mathsf{CH}}

\DeclareMathOperator{\add}{add}
\DeclareMathOperator{\cof}{cof}
\DeclareMathOperator{\cov}{cov}
\DeclareMathOperator{\non}{non}
\DeclareMathOperator{\diam}{diam}
\DeclareMathOperator{\Std}{Std}
\DeclareMathOperator{\succs}{succ}
\newcommand{\nor}{\mathsf{nor}}
\newcommand{\Loc}{\mathbf{Lc}}
\newcommand{\Sl}{\mathcal S}
\newcommand{\leT}{\preceq_{\mathrm T}}

\theoremstyle{plain}
\newtheorem{thm}{Theorem}
\newtheorem{lem}[thm]{Lemma}
\newtheorem{prop}[thm]{Proposition}
\newtheorem{cor}[thm]{Corollary}
\theoremstyle{definition}
\newtheorem{defi}[thm]{Definition}
\newtheorem{rmk}[thm]{Remark}
\setlist[enumerate,1]{label=(\arabic*),leftmargin=*}
\title{The additivity of a certain Hausdorff measure can differ from that of the Lebesgue measure}
\author{Tatsuya Goto}
\affil{Institute of Discrete Mathematics and Geometry, TU Wien\\
  Wiedner Hauptstrasse 8--10/104, 1040 Wien, Austria\\
  \texttt{goto.tatsuya@icloud.com}}
\date{\today}

\begin{document}
\maketitle

\begin{abstract}
Let $\mathcal{N}^h_\Omega$ be the null ideal of the Hausdorff measure constructed by Davies and Rogers.
We give a Tukey reduction of $(\mathcal{N}^h_\Omega,\mathcal{N}^h_\Omega,\subseteq)$ to a localization system with finite coordinate sets.
It follows that $\mathfrak{v}^\forall_{D,g}\le\operatorname{add}(\mathcal{N}^h_\Omega)$ and $\operatorname{cof}(\mathcal{N}^h_\Omega)\le\mathfrak{c}^\forall_{D,g}$ for some $D, g \in \omega^\omega$.
Consequently, we prove the consistency of $\mathfrak{d}<\operatorname{add}(\mathcal{N}^h_\Omega)$ and, separately, $\operatorname{cof}(\mathcal{N}^h_\Omega)<\mathfrak{b}$.
Hence the additivity and cofinality of $\mathcal{N}^h_\Omega$ can differ from those of the Lebesgue null ideal.
We also show that the finite graphs in the Davies--Rogers construction can be chosen so that $\operatorname{cov}(\mathcal{N}^h_\Omega)\le\operatorname{non}(\mathcal{N}^h_\Omega)$.
\end{abstract}


\section{Introduction}

Combining Howroyd's theorem \cite{howroyd1995dimension} with Fremlin's proof of \cite[Theorem~534B]{fremlin2008measure}, adapted to continuous doubling gauge functions, gives the following: if $X$ is an analytic metric space and $f$ is such a gauge with $\Haus^f_X(X)>0$, then its null ideal $\scrN^f_X$ satisfies
\[
  \add(\scrN^f_X)=\add(\scrN),\qquad
  \cof(\scrN^f_X)=\cof(\scrN).
\]
It is therefore natural to ask whether these equalities can fail for a Hausdorff measure whose gauge does not satisfy the doubling condition.
In this paper, we show that this is indeed possible: for the Hausdorff measure constructed by Davies and Rogers, each equality can consistently fail.

Davies and Rogers \cite{daviesrogers} constructed a compact metric space $\Omega$ and a gauge $h$ such that
\[
  \Haus^h_\Omega(\Omega)=\infty,
  \qquad \Haus^h_\Omega(A)\in\{0,\infty\}\quad(A\subseteq\Omega).
\]
Write $\Nho=\{A\subseteq\Omega:\Haus^h_\Omega(A)=0\}$.
In \cite{goto-dr-notes} we proved
\[
  \add(\scrN)\le\add(\Nho),\qquad
  \cof(\Nho)\le\cof(\scrN),
\]
and asked whether either inequality can be strict.
We answer both questions affirmatively, and refute the proposed bounds $\add(\Nho)\le\frakb$ and $\frakd\le\cof(\Nho)$ \cite[Question~24(1),(2)]{goto-dr-notes}.

For a proper ideal $\mathcal J$ on a set $X$ containing all singletons, put
\begin{align*}
  \add(\mathcal J)&=\min\{|\mathcal A|:\mathcal A\subseteq\mathcal J,
                                      \ \bigcup\mathcal A\notin\mathcal J\},\\
  \cov(\mathcal J)&=\min\{|\mathcal A|:\mathcal A\subseteq\mathcal J,
                                      \ \bigcup\mathcal A=X\},\\
  \non(\mathcal J)&=\min\{|A|:A\subseteq X,\ A\notin\mathcal J\},\\
  \cof(\mathcal J)&=\min\{|\mathcal A|:\mathcal A\subseteq\mathcal J,
                 \ (\forall A\in\mathcal J)(\exists B\in\mathcal A)\ A\subseteq B\}.
\end{align*}
For $x,y\in\omega^\omega$, write $x\le^*y$ if $x(n)\le y(n)$ for all but finitely many $n$.
The bounding number $\frakb$ and the dominating number $\frakd$ are the least sizes of an unbounded and a dominating family in $(\omega^\omega,\le^*)$, respectively, and $\frakc=2^{\aleph_0}$.
Let $\scrN$ denote the Lebesgue null ideal; see \cite{bartoszynski1995set} for background on these cardinal invariants.

A \emph{gauge function} is a nondecreasing, right-continuous function $f:[0,\infty)\to[0,\infty)$ with $f(0)=0$.
For a metric space $X$ and $A\subseteq X$, the $f$-Hausdorff measure is
\[
  \Haus^f_X(A)=\lim_{\delta\downarrow0}\inf\left\{
    \sum_{i<\omega}f(\diam U_i):
    A\subseteq\bigcup_{i<\omega}U_i,\quad
    U_i\subseteq X,\quad \diam U_i\le\delta
  \right\},
\]
and its null ideal is $\scrN^f_X=\{A\subseteq X:\Haus^f_X(A)=0\}$.
The gauge $f$ is \emph{doubling} if $f(2r)\le C f(r)$ for some $C>0$ and all sufficiently small $r>0$.

Our main result uses the localization system
\[
  \Loc(D,g)=\left(\prod_{n<\omega}D_n,\Sl(D,g),\in^*\right),
  \qquad \Sl(D,g)=\prod_{n<\omega}[D_n]^{\le g(n)},
\]
where $x\in^*\varphi$ means $x(n)\in\varphi(n)$ for all but finitely many $n$.
Its bounding and dominating numbers are denoted by $\frakv^\forall_{D,g}$ and $\frakc^\forall_{D,g}$.
Explicitly, $\frakv^\forall_{D,g}$ is the least size of a family in $\prod_nD_n$ that no single slalom in $\Sl(D,g)$ eventually localizes, and $\frakc^\forall_{D,g}$ is the least size of a family of slaloms that eventually localizes every member of $\prod_nD_n$.

\begin{thm}\label{thm:main}
For the finite sets $D_n$ defined in \eqref{eq:parameters} and $g(n)=2^n$,
\[
  (\Nho,\Nho,\subseteq)\leT\Loc(D,g).
\]
Consequently,
\[
  \frakv^\forall_{D,g}\le\add(\Nho),\qquad
  \cof(\Nho)\le\frakc^\forall_{D,g}.
\]
\end{thm}

The key step is a covering lemma: families $\mathcal F_n$ of rank $n$ standard sets satisfying $a_n|\mathcal F_n|\le2^n$ still have a null limsup.
We prove this by moving most of each family to rank $n-1$.

\begin{cor}\label{cor:consistency}
Assuming $\Con(\ZFC)$, each of the following is consistent with $\ZFC$:
\begin{align}
  \aleph_1=\frakb=\frakd
    &<\add(\Nho)=\frakc=\aleph_2;\label{eq:add-model}\\
  \aleph_1=\cof(\Nho)
    &<\frakb=\frakd=\frakc=\aleph_2.\label{eq:cof-model}
\end{align}
\end{cor}

We obtain \eqref{eq:cof-model} in the Laver model.
For \eqref{eq:add-model}, we iterate a proper $\omega^\omega$-bounding forcing which adds an eventual localizing slalom.
Since
\[
  \add(\scrN)\le\frakb\le\frakd\le\cof(\scrN),
\]
Corollary~\ref{cor:consistency} gives the following.

\begin{cor}\label{cor:lebesgue}
Assuming $\Con(\ZFC)$, each of the following is consistent with $\ZFC$:
\[
  \add(\scrN)<\add(\Nho),\qquad
  \cof(\Nho)<\cof(\scrN).
\]
\end{cor}

In Section~\ref{sec:covering-uniformity}, we use the Schrijver graphs and tournaments
from \cite{girao2024tournament} to choose the finite graphs in the Davies--Rogers construction so that
\[
  \cov(\Nho)\le\non(\Nho).
\]
Thus, for these Davies--Rogers Hausdorff measures, we cannot obtain the consistency of $\non(\Nho)<\cov(\Nho)$.
This is somewhat surprising, since one might expect a countable support iteration of the idealized forcing associated with $\Nho$ to yield a model of $\ZFC$ satisfying $
  \aleph_1=\non(\Nho)<\cov(\Nho)=\frakc=\aleph_2.
$
\section{Standard covers}

We recall the notation.
For each $m\ge1$, fix a finite graph $G(m)$ with chromatic number greater than $m$ such that
\begin{equation}\label{eq:weighted}
  \forall w:G(m)\to\R_{\ge0}\ \exists H\subseteq G(m)\quad
  \left[H\text{ is independent}\ \land\
  \sum_{v\in H}w(v)\ge\frac14\sum_{v\in G(m)}w(v)\right].
\end{equation}
Such graphs exist by \cite{daviesrogers}.
Set
\[
  M_0=1,\qquad G_n=G(M_n),\qquad N_n=|G_n|,\qquad M_{n+1}=2N_n,
  \qquad \Omega=\prod_{n<\omega}G_n.
\]
If $n$ is the first coordinate where $x\ne y$ differ, let
\[
  \rho(x,y)=
  \begin{cases}
    2^{-n+1},&\text{if $x(n)$ and $y(n)$ are adjacent},\\
    2^{-n},&\text{otherwise}.
  \end{cases}
\]
Put $h(0)=0$ and $h(2^{-n})=(\prod_{i\le n}M_i)^{-1}$, interpolate linearly on $[0,1]$, and set $h(r)=r$ for $r\ge1$.

For $s\in\prod_{i<n}G_i$ and a nonempty independent set $H\subseteq G_n$, write
\[
  S(s,H)=\{x\in\Omega:x\restrict n=s,\ x(n)\in H\},
  \qquad
  \Std_n=\{S(s,H):s,H\text{ as above}\}.
\]
These are the \emph{standard sets of rank $n$}.
They are clopen and have diameter $2^{-n}$.
Put
\begin{equation}\label{eq:standard-parameters}
  a_n=h(2^{-n})=\frac1{\prod_{i\le n}M_i},\qquad
  P_n=\prod_{i<n}N_i,\qquad P_na_n=2^{-n}.
\end{equation}
Every set with at least two points and diameter at most $1$ is contained in a standard set of the same diameter \cite{daviesrogers}.
A singleton is contained in standard sets of arbitrarily high rank.
Thus null sets admit standard covers $\mathcal U$ with arbitrarily small $\sum_{S\in\mathcal U}h(\diam S)$ and arbitrarily high minimum rank.

For $\mathcal F = \seq{\mathcal F_n : n \in \omega}$ with $\mathcal F_n \subseteq\Std_n$, put
\[
  N_{\mathcal F}
    =\limsup_n\bigcup\mathcal F_n
    =\bigcap_{m<\omega}\bigcup_{n\ge m}\bigcup\mathcal F_n.
\]
This is a $G_\delta$ set.
We shall use
\begin{equation}\label{eq:summable}
  \sum_n a_n|\mathcal F_n|<\infty
  \quad\Longrightarrow\quad N_{\mathcal F}\in\Nho.
\end{equation}
Indeed, the tail cover $\bigcup_{n\ge m}\mathcal F_n$ satisfies $\sum_{n\ge m}a_n|\mathcal F_n|\to0$ as $m\to\infty$, and every member has diameter at most $2^{-m}$.

\begin{lem}\label{lem:coding}
For every $A\in\Nho$, there is a sequence $\sigma = \seq{\sigma(n) : n \in \omega}$ with $\sigma(n) \subseteq\Std_n$ such that
\[
  A\subseteq N_\sigma,\qquad
  \sum_n a_n|\sigma(n)|\le1.
\]
\end{lem}

\begin{proof}
For each $j$, choose a countable standard cover $\mathcal U_j$ of $A$ such that
\[
  \mathcal U_j\subseteq\bigcup_{n\ge j}\Std_n,
  \qquad \sum_{S\in\mathcal U_j}h(\diam S)\le2^{-j-1}.
\]
Let $\sigma(n)=\Std_n\cap\bigcup_j\mathcal U_j$.
Then
\[
  \sum_n a_n|\sigma(n)|
    \le\sum_j\sum_{S\in\mathcal U_j}h(\diam S)\le1.
\]
Every $x\in A$ belongs to members of these covers of arbitrarily high rank, so $x\in N_\sigma$.
\end{proof}

\begin{lem}\label{lem:compression}
If $\mathcal F_n\subseteq\Std_n$ and $a_n|\mathcal F_n|\le2^n$ for all $n$, then $N_{\mathcal F}\in\Nho$.
\end{lem}

\begin{proof}
Fix $n\ge1$.
For each $t\in\prod_{i<n-1}G_i$, define
\[
  w_t(v)=\bigl|\{S(s,H)\in\mathcal F_n:s=t\append v\}\bigr|
  \qquad(v\in G_{n-1}).
\]
Apply \eqref{eq:weighted} successively $4n$ times, setting the weights of selected vertices to zero after each step.
This gives independent sets $K_{t,j}$, $j<4n$, with
\[
  \sum_{v\notin\bigcup_{j<4n}K_{t,j}}w_t(v)
    \le\left(\frac34\right)^{4n}\sum_v w_t(v).
\]
Put
\[
  \mathcal R_n
    =\{S(t\append v,H)\in\mathcal F_n:
        t\in\textstyle\prod_{i<n-1}G_i,
        v\in G_{n-1}\setminus\bigcup_{j<4n}K_{t,j}\},
\]
and
\[
  \mathcal Q_{n-1}
    =\{S(t,K_{t,j}):t\in\textstyle\prod_{i<n-1}G_i,\ j<4n,
                         \ K_{t,j}\ne\varnothing\}.
\]
Since
\[
  v\in K_{t,j}\quad\Longrightarrow\quad
  S(t\append v,H)\subseteq S(t,K_{t,j}),
\]
and since $P_na_n=2^{-n}$ implies $|\mathcal F_n|\le 2^n/a_n=4^nP_n$, we have
\begin{align*}
  \bigcup\mathcal F_n
    &\subseteq\bigcup\mathcal R_n\cup\bigcup\mathcal Q_{n-1},\\
  a_n|\mathcal R_n|
    &\le\left(\frac34\right)^{4n}a_n|\mathcal F_n|\\
    &\le\left(\frac34\right)^{4n}4^nP_na_n
      =\left(\frac34\right)^{4n}2^n
      =\left(\frac{81}{128}\right)^n,\\
  a_{n-1}|\mathcal Q_{n-1}|
    &\le4nP_{n-1}a_{n-1}=4n\,2^{-(n-1)}.
\end{align*}
Both bounds on the right are summable in $n$.
Hence \eqref{eq:summable} yields
\[
  N_{\mathcal F}
    \subseteq\limsup_{n\ge1}\bigcup\mathcal R_n
       \ \cup\ \limsup_{n\ge1}\bigcup\mathcal Q_{n-1}\in\Nho.
  \qedhere
\]
\end{proof}

\section{The Tukey reduction}\label{sec:tukey}

For relational systems $\mathbf R=(X,Y,R)$ and $\mathbf R'=(X',Y',R')$, a \emph{Tukey morphism} consists of maps
\[
  \Phi:X\to X',\qquad \Psi:Y'\to Y,
  \qquad \Phi(x)\mathrel{R'}y'\ \Longrightarrow\ x\mathrel R\Psi(y').
\]
We write $\mathbf R\leT\mathbf R'$ when such maps exist.
Then
\[
  \mathfrak b(\mathbf R')\le\mathfrak b(\mathbf R),\qquad
  \mathfrak d(\mathbf R)\le\mathfrak d(\mathbf R'),
\]
where $\mathfrak b$ and $\mathfrak d$ are the least sizes of an unbounded and a dominating family, respectively.
In particular,
\[
  \mathfrak b(\Nho,\Nho,\subseteq)=\add(\Nho),\qquad
  \mathfrak d(\Nho,\Nho,\subseteq)=\cof(\Nho).
\]

Fix the finite alphabets and the width
\begin{equation}\label{eq:parameters}
  D_n=\{\mathcal F\subseteq\Std_n:a_n|\mathcal F|\le1\},
  \qquad D=\seq{D_n:n<\omega},\qquad g(n)=2^n.
\end{equation}
Each $D_n$ contains the empty family and all singleton families.
Since $N_i\ge2$, we have
\begin{equation}\label{eq:nontrivial}
  |D_n|\ge1+|\Std_n|\ge1+P_nN_n>2^n=g(n).
\end{equation}
For finite alphabets we use the same localization notation as for their cardinalities, via fixed bijections.
Explicitly,
\begin{align*}
  \frakv^\forall_{D,g}
    &=\min\{|X|:X\subseteq\textstyle\prod_nD_n,
          \ (\forall\varphi\in\Sl(D,g))(\exists x\in X)\ x\notin^*\varphi\},\\
  \frakc^\forall_{D,g}
    &=\min\{|F|:F\subseteq\Sl(D,g),
          \ (\forall x\in\textstyle\prod_nD_n)(\exists\varphi\in F)\ x\in^*\varphi\}.
\end{align*}

\begin{proof}[Proof of Theorem~\ref{thm:main}]
For each $A\in\Nho$, choose $\sigma_A$ as in Lemma~\ref{lem:coding} and set $\Phi(A)=\sigma_A\in\prod_nD_n$.
For $\varphi\in\Sl(D,g)$, define
\[
  \mathcal F_n^\varphi=\bigcup_{\tau\in\varphi(n)}\tau,
  \qquad
  \Psi(\varphi)=B_\varphi
    :=\limsup_n\bigcup\mathcal F_n^\varphi.
\]
By \eqref{eq:parameters},
\[
  a_n|\mathcal F_n^\varphi|
    \le\sum_{\tau\in\varphi(n)}a_n|\tau|
    \le|\varphi(n)|\le2^n.
\]
Lemma~\ref{lem:compression} gives $\Psi(\varphi)\in\Nho$.
Finally,
\[
  \Phi(A)\in^*\varphi
    \quad\Longrightarrow\quad
  A\subseteq N_{\sigma_A}\subseteq B_\varphi=\Psi(\varphi).
\]
Thus $(\Phi,\Psi)$ is the required Tukey morphism, and both cardinal inequalities follow.
\end{proof}

\section{Small cofinality}

Recall that the Laver property implies, for any ground-model sequence of finite sets $D_n$ and any $g\in\omega^\omega$ with $g(n)\to\infty$,
\begin{equation}\label{eq:laver-property}
  \Vdash_{\Pp}\ (\forall x\in\textstyle\prod_nD_n)
    (\exists\varphi\in\Sl(D,g)^V)\ x\in^*\varphi.
\end{equation}
Here $\Sl(D,g)^V$ denotes the ground-model slaloms.

\begin{prop}\label{prop:laver}
If $\Pp$ has the Laver property, the ground-model Borel null codes remain cofinal in $\Nho$ after forcing with $\Pp$.
\end{prop}

\begin{proof}
In $V[G]$, let $A\in\Nho$.
Choose $\sigma\in\prod_nD_n$ with $A\subseteq N_\sigma$, and use \eqref{eq:laver-property} to find $\varphi\in\Sl(D,g)^V$ with $\sigma\in^*\varphi$.
Then
\[
  A\subseteq N_\sigma\subseteq B_\varphi\in\Nho.
\]
The code of $B_\varphi$ belongs to $V$.
\end{proof}

\begin{proof}[Proof of Corollary~\ref{cor:consistency}, \eqref{eq:cof-model}]
Start with the ground model $V$ satisfying $\CH$ and perform the countable support iteration of Laver forcing of length $\omega_2$.
Laver forcing is proper, has the Laver property, and adds a dominating real; the Laver property is preserved by countable support iterations of proper forcings \cite[Theorem~20.7 and Chapter~23, Related Result~132]{halbeisen2012}.
The resulting model satisfies
\[
  \frakb=\frakd=\frakc=\aleph_2.
\]
Indeed, every family of at most $\aleph_1$ reals is contained in an intermediate model, and the next Laver real dominates it.
By \eqref{eq:laver-property}, all ground-model slaloms together localize $\prod_nD_n$.
There are only $\aleph_1$ such slaloms, so Theorem~\ref{thm:main} gives
\[
  \aleph_1\le\cof(\Nho)\le\frakc^\forall_{D,g}\le\aleph_1.
\]
This proves \eqref{eq:cof-model}.
\end{proof}

\section{Large additivity}

It is well known that $\frakv^\forall_{D,g}$ can be increased by iterating a proper $\omega^\omega$-bounding tree forcing that adds an eventual localizing slalom.
For completeness, we give the construction and proof here, using this forcing to obtain \eqref{eq:add-model}.

In this section, $D_n$ may be any nonempty finite sets and $g\in\omega^\omega$ satisfies
\[
  g(n)\longrightarrow\infty,\qquad 1\le g(n)<|D_n|.
\]
Put $E_n=[D_n]^{\le g(n)}$.
For $X\subseteq E_n$, define
\[
  \nor_n(X)=\min\{|u|:u\subseteq D_n,
                  \ (\forall a\in X)\ u\nsubseteq a\}.
\]
This minimum exists since $g(n)<|D_n|$.
In particular, $\nor_n(\varnothing)=0$ and $\nor_n(E_n)=g(n)+1$.

\begin{lem}\label{lem:norm}
For $X,Y\subseteq E_n$ and $v\in D_n$,
\begin{align}
  \nor_n(X\cup Y)&\le\nor_n(X)+\nor_n(Y),\label{eq:subadd}\\
  \nor_n(\{a\in X:v\in a\})&\ge\nor_n(X)-1.\label{eq:prune}
\end{align}
\end{lem}

\begin{proof}
For \eqref{eq:subadd}, take the union of witnesses for $X$ and $Y$.
For \eqref{eq:prune}, if $u$ witnesses the left-hand side, then $u\cup\{v\}$ witnesses an upper bound for $\nor_n(X)$.
\end{proof}

\begin{defi}
The forcing $\Q(D,g)$ consists of the nonempty trees without terminal nodes
\[
  T\subseteq\bigcup_{m<\omega}\prod_{n<m}E_n
\]
such that, writing
\[
  \succs_T(t)=\{a\in E_{|t|}:t\append a\in T\},\qquad
  \nu_T(m)=\min_{t\in T\cap\prod_{n<m}E_n}
                         \nor_m(\succs_T(t)),
\]
we have $\nu_T(m)\to\infty$.
The order is inclusion.
For $s\in T$, let $T\restrict s$ consist of the nodes of $T$ comparable with $s$.
\end{defi}

The full tree is a condition, and $T\restrict s\in\Q(D,g)$ whenever $s\in T$.
Restrictions to levels of arbitrarily large height show that a generic filter determines a branch $\dot\varphi\in\prod_nE_n=\Sl(D,g)$.

\begin{lem}\label{lem:localizes}
The generic branch eventually localizes every ground-model real:
\[
  \Vdash_{\Q(D,g)}\ (\forall x\in(\textstyle\prod_nD_n)^V)
                         \ x\in^*\dot\varphi.
\]
\end{lem}

\begin{proof}
Fix $x\in(\prod_nD_n)^V$ and $T\in\Q(D,g)$.
Choose $m_0$ such that $\nu_T(m)\ge2$ for $m\ge m_0$.
At each such level, retain only successors $a$ with $x(m)\in a$.
By \eqref{eq:prune}, the resulting tree $S$ has no terminal nodes and
\[
  \nu_S(m)\ge\nu_T(m)-1\longrightarrow\infty.
\]
Thus $S\le T$ forces $(\forall m\ge m_0)\,x(m)\in\dot\varphi(m)$.
\end{proof}

\begin{lem}\label{lem:finite}
Let $k\ge1$, $T\in\Q(D,g)$, and $s\in T$.
Suppose
\[
  \nor_{|t|}(\succs_T(t))\ge2k\qquad(s\subseteq t\in T).
\]
For every dense $\mathcal D\subseteq\Q(D,g)$, there are $S\le T\restrict s$ and a finite $\mathcal A\subseteq\mathcal D$ such that $s\in S$, $\mathcal A$ is predense below $S$, and
\[
  \nor_{|t|}(\succs_S(t))\ge k\qquad(s\subseteq t\in S).
\]
\end{lem}

\begin{proof}
Call $t\supseteq s$ \emph{good} if there are a condition $S_t\le T\restrict t$ containing $t$ and a finite $\mathcal A_t\subseteq\mathcal D$ predense below $S_t$, with all norms above $t$ at least $k$.
The requirement $S_t\in\Q(D,g)$ includes divergence of its level norms.

Let $G_t=\{a\in\succs_T(t):t\append a\text{ is good}\}$.
Finite unions of the witnessing trees give
\[
  \nor_{|t|}(G_t)\ge k\quad\Longrightarrow\quad t\text{ is good}.
\]
If $s$ is not good, retain only non-good successors above $s$, obtaining $T^-$.
For every remaining $t\supseteq s$, \eqref{eq:subadd} gives
\[
  \nor_{|t|}(\succs_{T^-}(t))
    \ge\nor_{|t|}(\succs_T(t))-(k-1).
\]
Hence $T^-\in\Q(D,g)$.
Choose $R\le T^-$ in $\mathcal D$, and take $t\supseteq s$ sufficiently high in $R$ that all norms of $R$ above $t$ are at least $k$.
Then $R\restrict t$, together with $\{R\}$, witnesses that $t$ is good, a contradiction.
\end{proof}

\begin{prop}\label{prop:proper-bounding}
$\Q(D,g)$ is proper and $\omega^\omega$-bounding.
\end{prop}

\begin{proof}
We first describe the fusion.
Given $T_0$ and dense sets $\seq{\mathcal D_j:j<\omega}$, recursively choose $T_{j+1}\le T_j$ and strictly increasing heights $\ell_j$ with
\[
  \nu_{T_j}(m)\ge2(j+1)\qquad(m\ge\ell_j).
\]
Apply Lemma~\ref{lem:finite} with $k=j+1$ below every node at level $\ell_j$, and take the union of the resulting trees.
Then
\begin{align*}
  T_{j+1}\cap\bigcup_{m\le\ell_j}\prod_{n<m}E_n
    &=T_j\cap\bigcup_{m\le\ell_j}\prod_{n<m}E_n,\\
  \nu_{T_{j+1}}(m)&\ge j+1\qquad(m\ge\ell_j),
\end{align*}
and some finite $\mathcal A_j\subseteq\mathcal D_j$ is predense below $T_{j+1}$.
Therefore
\[
  T_\infty=\bigcap_jT_j\in\Q(D,g),\qquad
  \nu_{T_\infty}(m)\ge j+1
       \quad(\ell_j\le m<\ell_{j+1}).
\]

For bounding, let $\dot f$ be a name for an element of $\omega^\omega$ and take $\mathcal D_j$ to be the conditions deciding $\dot f(j)$.
Writing $F_j$ for the finite set of values decided by $\mathcal A_j$, we obtain
\[
  T_\infty\Vdash (\forall j)\ \dot f(j)\in F_j.
\]
Thus $j\mapsto\max F_j$ is a ground-model bound.

For properness, fix a countable $M\prec H_\theta$ with $\Q(D,g),T_0\in M$, and enumerate the dense subsets of $\Q(D,g)$ belonging to $M$ as $\seq{\mathcal D_j:j<\omega}$.
Carry out the same construction with $T_j,\mathcal A_j\in M$ at each finite stage.
Since $\mathcal A_j$ is finite, $\mathcal A_j\subseteq\mathcal D_j\cap M$.
Consequently, every $\mathcal D_j\cap M$ is predense below $T_\infty$, so $T_\infty$ is $M$-generic.
\end{proof}

\begin{proof}[Proof of Corollary~\ref{cor:consistency}, \eqref{eq:add-model}]
Use the parameters from \eqref{eq:parameters}; \eqref{eq:nontrivial} verifies the hypotheses on $D,g$.
Starting with $\CH$, form the countable support iteration
\[
  \seq{\Pp_\alpha,\dot\Q_\alpha:\alpha<\omega_2},\qquad
  \Vdash_{\Pp_\alpha}\ \dot\Q_\alpha=\Q(D,g).
\]
Every iterand has size at most $\frakc$ and adds a real: a ground-model slalom cannot localize all ground-model members of $\prod_nD_n$, since $|D_n|>g(n)$.
By the standard iteration theorems \cite[Theorems~20.3, 20.5, and 20.6]{halbeisen2012}, cardinals are preserved and
\[
  V[G_{\omega_2}]\models
  \frakb=\frakd=\aleph_1,\qquad \frakc=\aleph_2.
\]
Here the ground-model dominating family remains dominating by $\omega^\omega$-bounding.

Let $X\subseteq\prod_nD_n$ have size at most $\aleph_1$ in the final model.
For some $\alpha<\omega_2$, $X\subseteq V[G_\alpha]$.
Lemma~\ref{lem:localizes} gives
\[
  (\forall x\in X)\ x\in^*\varphi_\alpha,
\]
where $\varphi_\alpha$ is the slalom added at stage $\alpha$.
This relation remains true in the final model.
Hence
\[
  \aleph_2\le\frakv^\forall_{D,g}
      \le\add(\Nho)\le\frakc=\aleph_2,
\]
by Theorem~\ref{thm:main}.
This proves \eqref{eq:add-model}.
\end{proof}

\section{Covering and uniformity}\label{sec:covering-uniformity}

We choose the graphs $G(m)$ so that $\cov(\Nho)\le\non(\Nho)$.

\begin{lem}\label{lem:relations-covering}
For a graph $G$ and $H\subseteq V(G)$, let $G[H]$ denote the subgraph of $G$ induced by $H$.
Write $\chi(G)$ for the chromatic number of $G$.
Suppose that $R_n\subseteq G_n\times G_n$ and $r_n\in\omega$ satisfy
\[
  uR_nv\ \lor\ vR_nu\quad(u,v\in G_n),\qquad
  \chi(G_n[R_n(u)])\le r_n\quad(u\in G_n),\qquad
  \sum_n r_n2^{-n}<\infty,
\]
where $R_n(u)=\{v\in G_n:uR_nv\}$.
Then there is a $G_\delta$ set $B\subseteq\Omega^2$ such that
\begin{equation}\label{eq:null-sections}
  B_x\in\Nho,\qquad \Omega\setminus B^y\in\Nho
  \quad(x,y\in\Omega),
\end{equation}
where $B_x=\{y:(x,y)\in B\}$ and $B^y=\{x:(x,y)\in B\}$.
Consequently, $\cov(\Nho)\le\non(\Nho)$.
\end{lem}

\begin{proof}
Define
\[
  B=\{(x,y)\in\Omega^2:(\exists^\infty n)\ x(n)R_ny(n)\}.
\]
Each coordinate relation is clopen, so $B$ is $G_\delta$.
Fix $x\in\Omega$ and partition each $R_n(x(n))$ into nonempty independent
sets $H_{n,j}$, $j<\ell_n\le r_n$.
Put
\[
  \mathcal F_n(x)=
    \{S(s,H_{n,j}):s\in\textstyle\prod_{i<n}G_i,\ j<\ell_n\}.
\]
By \eqref{eq:standard-parameters},
\[
  B_x=N_{\mathcal F(x)},\qquad
  a_n|\mathcal F_n(x)|\le a_nP_nr_n=r_n2^{-n}.
\]
Thus $B_x\in\Nho$ by \eqref{eq:summable}.
Moreover,
\[
  x\notin B^y
  \ \Longrightarrow\ (\forall^\infty n)\ \neg(x(n)R_ny(n))
  \ \Longrightarrow\ (\forall^\infty n)\ y(n)R_nx(n)
  \ \Longrightarrow\ x\in B_y.
\]
Hence $\Omega\setminus B^y\subseteq B_y\in\Nho$.
Choose $X\subseteq\Omega$ with $X\notin\Nho$ and $|X|=\non(\Nho)$.
For every $y\in\Omega$, \eqref{eq:null-sections} gives $X\cap B^y\ne\varnothing$.
Therefore
\[
  \Omega=\bigcup_{x\in X}B_x,\qquad
  \cov(\Nho)\le|X|=\non(\Nho).\qedhere
\]
\end{proof}

\begin{lem}\label{lem:graphs-covering}
For each $m\ge1$, there are a finite graph $G$ satisfying
$\chi(G)>m$ and \eqref{eq:weighted}, a relation $R\subseteq G\times G$, and
independent sets $H(u)\subseteq G$, $u\in G$, such that
\begin{align}
  &uRv\ \lor\ vRu,\qquad \chi(G[R(u)])\le3
       &&(u,v\in G),\label{eq:graph-relation}\\
  &u\in H(u),\qquad
    |\{v\in G:u\in H(v)\}|\ge\tfrac14|G|
       &&(u\in G).\label{eq:graph-independent}
\end{align}
\end{lem}

\begin{proof}
Fix an integer $c\ge\max\{3,m+1\}$.
By the construction in \cite[Theorem~2 and its proof]{girao2024tournament},
for all sufficiently large $k$ the Schrijver graph
$S=\operatorname{SG}(2k+c-2,k)$ admits a tournament $T$ on its vertex set with
\[
  \chi(S)=c,\qquad
  \chi(S[N_T^+(A)])\le2\quad(A\in V(S)),
\]
where $A\longrightarrow_T B$ means that $T$ has a directed edge from $A$ to $B$,
and $N_T^+(A)=\{B:A\longrightarrow_T B\}$.
Choose such a $k>c-2$ and put $q=2k+c-2$.
Write $u\sim_G v$ when $u$ and $v$ are adjacent in a graph $G$.
Explicitly,
\[
  \begin{gathered}
    V(S)=\{A\subseteq\{0,\ldots,q-1\}:|A|=k,
                  \ (\forall r\in A)\ (r+1)\bmod q\notin A\},\\
    A\sim_S B\ \Longleftrightarrow\ A\cap B=\varnothing.
  \end{gathered}
\]
Define
\[
  V(G)=V(S)\times\{0,\ldots,k-1\},\qquad
  (A,i)\sim_G(B,j)\ \Longleftrightarrow\ A\cap B=\varnothing.
\]
Projection onto $S$ and the induced copy $V(S)\times\{0\}$ give
$\chi(G)=\chi(S)=c>m$.
For $r<q$, the set
\[
  J_r=\{(A,i)\in V(G):r\in A\}
\]
is independent. For every $w:V(G)\to\R_{\ge0}$,
\[
  \frac1q\sum_{r<q}\sum_{u\in J_r}w(u)
    =\frac{k}{q}\sum_{u\in V(G)}w(u),\qquad
  \frac{k}{q}>\frac13>\frac14,
\]
so some $J_r$ witnesses \eqref{eq:weighted}.

Write $A=\{a_0<\cdots<a_{k-1}\}$ and set $H(A,i)=J_{a_i}$.
Then $(A,i)\in H(A,i)$.
Cyclic rotation of $\{0,\ldots,q-1\}$ preserves $V(S)$, whence
\[
  |\{A\in V(S):r\in A\}|=\frac{k}{q}|V(S)|\quad(r<q).
\]
For each $(B,j)\in V(G)$, it follows that
\begin{align*}
  |\{(A,i)\in V(G):(B,j)\in H(A,i)\}|
    &=\sum_{r\in B}|\{A\in V(S):r\in A\}|\\
    &=\frac{k^2}{q}|V(S)|
      =\frac{k}{q}|V(G)|>\frac14|V(G)|.
\end{align*}
Thus \eqref{eq:graph-independent} holds.
Finally, put
\[
  (A,i)R(B,j)\ \Longleftrightarrow\ A=B\ \lor\ A\longrightarrow_T B.
\]
Since $T$ is a tournament, $uRv\lor vRu$ for all $u,v\in V(G)$, including $u=v$.
Also,
\[
  R(A,i)=(\{A\}\cup N_T^+(A))\times\{0,\ldots,k-1\},\qquad
  \chi(G[R(A,i)])\le1+\chi(S[N_T^+(A)])\le3.\qedhere
\]
\end{proof}

\begin{thm}\label{thm:covering-uniformity}
There is a choice of the graphs $G(m)$ in the Davies--Rogers construction,
satisfying \eqref{eq:weighted}, \eqref{eq:graph-relation}, and
\eqref{eq:graph-independent}, for which
\[
  \cov(\Nho)\le\non(\Nho).
\]
\end{thm}

\begin{proof}
Choose $G(m)$ and $R(m)$ as in Lemma~\ref{lem:graphs-covering}.
With $G_n=G(M_n)$ and $R_n=R(M_n)$, Lemma~\ref{lem:relations-covering} applies:
\[
  r_n=3\quad(n<\omega),\qquad \sum_{n<\omega}3\cdot2^{-n}=6<\infty.\qedhere
\]
\end{proof}

\begin{rmk}
We do not know whether $\cov(\Nho)\le\non(\Nho)$ holds for every
Hausdorff measure obtained by the Davies--Rogers construction.
\end{rmk}

\section*{Declaration of generative AI}
During the preparation of this work, the author used OpenAI Codex with
the \texttt{gpt-5.6-sol} and \texttt{gpt-6-astra} models to assist with
literature searches, the discovery and development of proofs, language
editing, and \LaTeX{} drafting. The author independently checked,
reviewed, and revised all AI-assisted material and takes full
responsibility for the content of the paper.

\printbibliography

@article{fremlin2008measure,
	title={Measure theory, vol. 5},
	author={Fremlin, D. H.},
	journal={Set-Theoretic Measure Theory, Parts I, II. Torres Fremlin, Colchester},
	year={2008}
}

@article{howroyd1995dimension,
	title={On dimension and on the existence of sets of finite positive Hausdorff measure},
	author={Howroyd, J. D.},
	journal={Proceedings of the London Mathematical Society},
	volume={3},
	number={3},
	pages={581--604},
	year={1995},
	publisher={Wiley Online Library}
}

@article{daviesrogers,
	author = {Davies, R. O. and Rogers, C. A.},
	title = {The Problem of Subsets of Finite Positive Measure},
	journal = {Bulletin of the London Mathematical Society},
	volume = {1},
	number = {1},
	pages = {47-54},
	year = {1969},
	month = {03},
	issn = {0024-6093},
	doi = {10.1112/blms/1.1.47},
	url = {https://doi.org/10.1112/blms/1.1.47},
	eprint = {https://academic.oup.com/blms/article-pdf/1/1/47/786704/1-1-47.pdf},
}

@book{bartoszynski1995set,
	title={Set Theory: on the structure of the real line},
	author={Bartoszynski, T. and Judah, H.},
	year={1995},
	publisher={CRC Press}
}

@unpublished{goto-dr-notes,
  author = {Goto, Tatsuya},
  title = {The {Hausdorff} measure due to {Davies} and {Rogers} and its cardinal invariants},
  note = {arXiv preprint arXiv:2601.20414},
  year = {2026}
}

@article{girao2024tournament,
  author = {Gir\~ao, Ant\'onio and Hendrey, Kevin and Illingworth, Freddie
            and Lehner, Florian and Michel, Lukas and Savery, Michael
            and Steiner, Raphael},
  title = {Chromatic number is not tournament-local},
  journal = {Journal of Combinatorial Theory, Series B},
  volume = {168},
  pages = {86--95},
  year = {2024},
  doi = {10.1016/j.jctb.2024.04.005}
}

@book{halbeisen2012,
  author = {Halbeisen, Lorenz J.},
  title = {Combinatorial Set Theory: With a Gentle Introduction to Forcing},
  series = {Springer Monographs in Mathematics},
  publisher = {Springer},
  address = {London},
  year = {2012},
  doi = {10.1007/978-1-4471-2173-2},
  url = {https://people.math.ethz.ch/~halorenz/publications/pdf/cst.pdf}
}
\end{document}